\documentclass[11pt]{amsart}	%This version: JMJU Feb 16, 2023

\usepackage[utf8]{inputenc}
\usepackage{graphicx}
\usepackage{subcaption}
\usepackage{amsmath}
\usepackage{amssymb}
\usepackage{amsthm}
\usepackage{ifthen} 
\usepackage{color}

\usepackage{dsfont}

\def\polk#1{\setbox0=\hbox{#1}{\ooalign{\hidewidth\lower1.5ex\hbox{`}\hidewidth\crcr\unhbox0}}}

\newtheorem{Theorem}{Theorem}
\newtheorem{Definition}{Definition}
\newtheorem{Lemma}{Lemma}
\newtheorem{Corollary}{Corollary}
\newtheorem{Proposition}{Proposition}

\theoremstyle{definition}

\theoremstyle{remark}
\newtheorem{Remark}{Remark}

\allowdisplaybreaks

\newcommand{\intav}[1]{\mathchoice {\mathop{\vrule width 6pt height 3 pt depth -2.5pt\kern -8pt \intop}\nolimits_{\kern -6pt#1}} {\mathop{\vrule width5pt height 3 pt depth -2.6pt \kern -6pt \intop}\nolimits_{#1}}{\mathop{\vrule width 5pt height 3 pt depth -2.6pt \kern -6pt\intop}\nolimits_{#1}} {\mathop{\vrule width 5pt height 3 pt depth-2.6pt \kern -6pt \intop}\nolimits_{#1}}}

\usepackage{setspace}
\numberwithin{equation}{section}

\title[A degenerate transmission problem]{BMO--regularity for a degenerate transmission problem}

\author[V. Bianca]{Vincenzo Bianca}
\address{University of Coimbra, CMUC, Department of Mathematics, 3001-501 Coimbra, Portugal}{}
\email{vincenzo@mat.uc.pt}

\author[E. A. Pimentel]{Edgard A. Pimentel}
\address{University of Coimbra, CMUC, Department of Mathematics, 3001-501 Coimbra, Portugal and Pontifical Catholic University of Rio de Janeiro -- PUC-Rio, 22451-900 G\'avea, Rio de Janeiro-RJ, Brazil}{}
\email{edgard.pimentel@mat.uc.pt}

\author[J.M.~Urbano]{Jos\'{e} Miguel Urbano}
\address{King Abdullah University of Science and Technology (KAUST), Computer, Electrical and Mathematical Sciences and Engineering Division (CEMSE), Thuwal 23955-6900, Saudi Arabia and University of Coimbra, CMUC, Department of Mathematics, 3001-501 Coimbra, Portugal}{} 
\email{miguel.urbano@kaust.edu.sa} 

\date{\today}

\begin{document}

\begin{abstract}
We examine a transmission problem driven by a degenerate quasilinear operator with a natural interface condition. Two aspects of the problem entail genuine difficulties in the analysis: the absence of representation formulas for the operator and the degenerate nature of the diffusion process. Our arguments circumvent these difficulties and lead to new regularity estimates. For bounded interface data, we prove the local boundedness of weak solutions and establish an estimate for their gradient in ${\rm BMO}-$spaces. The latter implies solutions are of class $C^{0,{\rm Log-Lip}}$ across the interface. Relaxing the assumptions on the data, we establish local H\"older continuity for the solutions.
\end{abstract}



\maketitle

\section{Introduction}\label{sec_mollybloom}

Transmission problems describe diffusive processes within heterogeneous media that change abruptly across certain interfaces. They find application, for example, in the study of electromagnetic conductivity and composite materials, and their mathematical formulation involves a domain split into sub-regions, where partial differential equations (PDEs) are prescribed. Since the PDEs vary from region to region, the problem may have discontinuities across the interfaces. 
Consequently, the geometry of these interfaces (which, in contrast to free boundary problems, are fixed and given a priori) and the structure of the underlying equations play a crucial role in analysing transmission problems.

This class of problems first appeared circa 1950, in the work of Mauro Picone \cite{Picone1954}, as an attempt to address heterogeneous materials in elasticity theory. Several subsequent works developed the basics of the theory and generalised it in various directions \cite{Borsuk1968, Campanato1957, Campanato1959, Campanato1959a, Iliin-Shismarev1961, Schechter1960, Sheftel1963, Stampacchia1956}. We refer the interested reader to \cite{Borsuk2010} for a comprehensive account of this literature.

Developments concerning the regularity of the solutions to transmission problems are much more recent. In \cite{Li-Vogelius2000}, the authors study a class of elliptic equations in divergence form, with discontinuous coefficients, modelling composite materials with closely spaced interfacial boundaries, such as fibre-reinforced structures. The main result in that paper is the local H\"older continuity for the gradient of the solutions, with estimates. The findings in \cite{Li-Vogelius2000} are relevant from the applied perspective since the gradient of a solution accounts for the stress of the material, and estimating it shows the stresses remain uniformly bounded, even when fibres are arbitrarily close to each other. The vectorial counterpart of the results in \cite{Li-Vogelius2000} appeared in \cite{Li-Nirenberg2003}, where regularity estimates for higher-order derivatives of the solutions are obtained. See also the developments reported in \cite{Bonnetier2000}.

A further layer of analysis concerns the proximity of sub-regions in limiting scenarios. In \cite{Bao-Li-Yin1}, the authors examine a domain containing two subregions, which are $\varepsilon-$apart, for some $\varepsilon>0$. Within each sub-region, the diffusion process is given by a divergence-form equation with a diffusivity coefficient $A \neq 1$. In the remainder of the domain, the diffusivity is also constant but equal to $1$. By setting $A=+\infty$, the authors examine the case of perfect conductivity. The remarkable fact about this model is that estimates on the gradient of the solutions deteriorate as the two regions approach each other. In \cite{Bao-Li-Yin1}, the authors obtain blow-up rates for the gradient norm in terms of $\varepsilon\to 0$. We also notice the findings reported in \cite{Bao-Li-Yin2} extend those results to the context of multiple inclusions and also treat the case of perfect insulation $A=0$. We also refer the reader to \cite{Briane}.

More recently, the analysis of transmission problems focused on the geometry of the interface. The minimum requirements on the transmission interface yielding regularity properties for the solutions are particularly interesting. In \cite{CSCS2021}, the authors consider a domain split into two sub-regions. Inside each sub-region, the problem's solution is required to be a harmonic function, and a flux condition is prescribed along the interface separating the sub-regions. By resorting to a representation formula for harmonic functions, the authors establish the existence of solutions to the problem and prove that solutions are of class $C^{0,{\rm Log-Lip}}$ across the interface. In addition, under the assumption that the interface is locally of class $C^{1,\alpha}$, they prove the solutions are of class $C^{1,\alpha}$ within each sub-region, \emph{up to the transmission interface}. This fact follows from a new stability result allowing the argument to import information from the case of flat interfaces. In \cite{SCS2022}, the authors extend the analysis in \cite{CSCS2021} to the context of fully nonlinear elliptic operators. Under the assumption that the interface is of class $C^{1,\alpha}$, they prove that solutions are of class $C^{0,\alpha}$ across, and $C^{1,\alpha}$ up to the interface. Furthermore, if the interface is of class $C^{2,\alpha}$, then solutions became $C^{2,\alpha}-$regular, also up to the interface. The findings in \cite{SCS2022} rely on a new Aleksandrov-Bakelman-Pucci estimate and variants of the maximum principle and the Harnack inequality. We also notice the developments reported in \cite{Borsuk2019}. In that paper, the author proves local boundedness in a neighbourhood of boundary points for a transmission problem driven by a $p-$Laplacian type operator.

Our gist in this paper is to extend the results of \cite{CSCS2021} to the case of degenerate quasilinear equations, which are ``\textit{the natural, and, in a sense, the best generalisation of the $p-$Laplace equation}" (cf. \cite{Lieberman_1991}), namely
$$ \textnormal{div}\left(\frac{g\left(| D  u|\right)}{| D  u|} D  u\right)=0,$$
where $g$ is a nonlinearity satisfying appropriate assumptions. We first prove that weak solutions to the transmission problem, properly defined and whose existence follows from well-known methods, are locally bounded. The proof combines delicate inequalities with the careful choice of auxiliary test functions and a cut-off argument to produce a variant of the weak Harnack inequality. Working under a $C^1$ interface geometry, we then obtain an integral estimate for the gradient, leading to regularity in ${\rm BMO}-$spaces. As a corollary, we infer that solutions are of class $C^{0,{\rm Log-Lip}}$ across the fixed interface. For the particular case of the $p-$Laplace operator, this result follows directly from potential estimates obtained in \cite{DM2011, KM2014c}; we also refer to \cite{M2011,M2011a}. Finally, we relax the boundedness assumption on the interface data and derive local H\"older continuity estimates.

This transmission problem driven by a quasilinear degenerate operator presents genuine difficulties compared to the Laplacian's linear case. Firstly, the operator lacks representation formulas, and the strategy developed in \cite{CSCS2021} is no longer available. Secondly, the degenerate nature of the problem rules out the approach put forward in \cite{SCS2022}. Consequently, one must develop new machinery to examine the regularity of the solutions.

Another fundamental question in transmission problems concerns the optimal regularity \emph{up to the interface}. As mentioned before, results of this type appear in the recent works \cite{CSCS2021} and \cite{SCS2022}; see also \cite{dong20}. The issue remains open in the context of quasilinear degenerate problems, particularly for the $p-$Laplace operator. We believe the analysis of the boundary behaviour of $p-$harmonic functions may yield helpful information in this direction. 

The remainder of this article is organised as follows. Section \ref{sec_vicosa} contains the precise formulation of the problem, comments on the existence of a unique solution and gathers basic material used in the paper. In Section \ref{sec_alkhawarizmi}, we put forward the proof of the local boundedness. The proof of the BMO--regularity and its consequences is the object of Section \ref{sec_beacon}, where further generalisations are also included.

\section{Setting of the problem and auxiliary results}\label{sec_vicosa}

In this section, we precisely state our transmission problem, introduce the notion of a weak solution and comment on its existence and uniqueness. We then collect several auxiliary results.

\subsection{Problem setting and assumptions}
Let $\Omega\subset\mathbb{R}^d$ be a bounded domain and fix $\Omega_1\Subset\Omega$. Define $\Omega_2:=\Omega\setminus\overline{\Omega_1}$ and consider the interface $\Gamma:=\partial\Omega_1$, which we assume is a $(d-1)-$surface of class $C^1$. For a function $u:\overline{\Omega}\to\mathbb{R}$, we set 
\begin{equation*}
u_1:=u\big|_{\overline{\Omega_1}}\hspace{.3in} \mbox{and}\hspace{.3in} u_2:=u\big|_{\overline{\Omega_2}}.
\end{equation*}
Note that we necessarily have $u_1=u_2$ on $\Gamma$. Denoting with $\nu$  the unit normal vector to $\Gamma$ pointing inwards to $\Omega_1$, we write 
$$\frac{\partial u_i}{\partial\nu} = Du_i \cdot \nu, \quad i=1,2.$$
%(u_i)_{\nu}=
For a nonlinearity $g$, satisfying appropriate assumptions, we consider the quasilinear degenerate transmission problem consisting of finding a function $u:\overline{\Omega}\to\mathbb{R}$ such that
\begin{equation}\label{eq_stima118}
\begin{cases}
		\textnormal{div}\left(\frac{g\left(| D  u_1|\right)}{| D  u_1|} D  u_1\right)=0&\hspace{.2in}\mbox{in}\hspace{.2in}\Omega_1\\
		\vspace*{-0.3cm}\\
		\textnormal{div}\left(\frac{g\left(| D  u_2|\right)}{| D  u_2|} D  u_2\right)=0&\hspace{.2in}\mbox{in}\hspace{.2in}\Omega_2,\\
\end{cases}
\end{equation}
with the additional conditions
\begin{equation}\label{eq_stima119}
\begin{cases}
u=0&\hspace{.2in}\mbox{on}\hspace{.2in}\partial\Omega\\
\vspace*{-0.3cm}\\
\frac{g(| D  u_1|)}{| D  u_1|}\frac{\partial u_1}{\partial\nu}-\frac{g(| D  u_2|)}{| D  u_2|}\frac{\partial u_2}{\partial\nu}=f&\hspace{.2in}\mbox{on}\hspace{.2in}\Gamma,
\end{cases}
\end{equation}
for a given function $f$.
%For $p>2$ and $f$ in a suitable function space, w

%To properly formulate the quasilinear degenerate transmission problem, we introduce some auxiliary functions. 

We assume the function $g\in C^1\left(\mathbb{R}^+_0\right)$ is such that
	\begin{equation}
		g_0\le\frac{tg'(t)}{g(t)}\le g_1,\quad\forall t>0,
	\label{business class}
	\end{equation}
	for fixed constants $1\le g_0\le g_1$. Moreover, we assume the monotonicity inequality 
	\begin{equation}\label{monicavitti}
		\bigg(\frac{g(|\xi|)}{|\xi|}\xi-\frac{g(|\zeta|)}{|\zeta|}\zeta\bigg)\cdot(\xi-\zeta)\ge C|\xi-\zeta|^p, \quad\forall\xi,\zeta\in\mathbb{R}^d,
	\end{equation}
holds for a certain $p>2$ and $C>0$. 
	
	By choosing $g(t)=t^{p-1}$, with $p>2$, one gets in \eqref{eq_stima118} two degenerate $p-$Laplace equations. A different example of a nonlinearity $g=g(t)$ satisfying \eqref{business class}-\eqref{monicavitti} is
	\[
		g(t):=t^{p-1}\ln\left(a+t\right)^\alpha,
	\]
	for $p>2$, $a>1$ and $\alpha>0$. 
	
We now define the primitive of $g$, 
	\begin{equation*}
		G(t)=\int_0^tg(s)\,{\rm d}s, \quad t \geq 0.
	\end{equation*}
Due to the assumptions on $g$, one concludes that $G:\mathbb{R}\to\mathbb{R}\cup\left\lbrace+\infty\right\rbrace$ is left-continuous and convex, or a \emph{Young function} (see \cite[Definition 3.2.1]{at1}).  Before proceeding, we introduce the Orlicz-Sobolev space defined by $G$.

\begin{Definition}[Orlicz-Sobolev space]\label{notto}
Let $G$ be a Young function. We define the Orlicz-Sobolev space $W^{1,G}(\Omega)$ as the set of weakly differentiable functions $u\in W^{1,1}(\Omega)$ such that 
\[
	\int_\Omega G\left(|u(x)|\right){\rm d}x+\int_\Omega G\left(|Du(x)|\right){\rm d}x<\infty.
\]
The space $W^{1,G}_0(\Omega)$ is the closure of $C^{\infty}_c(\Omega)$ in $W^{1,G}(\Omega)$. 
\end{Definition}

\subsection{Weak solutions}
The precise definition of solution we have in mind is the object of the following definition.

\begin{Definition}\label{def_weaksol}
A function $u\in W_0^{1,G}(\Omega)$ is a weak solution of \eqref{eq_stima118}-\eqref{eq_stima119} if 
\begin{equation}
\int_{\Omega}\frac{g\left(| D  u|\right)}{| D  u|} D  u\cdot D  v\,{\rm d}x=\int_{\Gamma}f v\,{\rm d}\mathcal{H}^{d-1},\hspace{.2in}\forall\,v\in W^{1,G}_0(\Omega).
\label{portia}
\end{equation}
\end{Definition}

We use the Hausdorff measure $\mathcal{H}^{d-1}$ in the surface integral to emphasise the operator acting on the solution is a measure supported along the interface, and we write
\begin{equation}\label{eq_pde}
	-\textnormal{div}\left(\frac{g\left(| D  u|\right)}{| D  u|} D  u\right)=f\,{\rm d}\mathcal{H}^{d-1}\big|_{\Gamma}.
\end{equation}

To justify \eqref{eq_pde}, we multiply both equations in \eqref{eq_stima118} by a test function $\varphi\in C^\infty_c(\Omega)$, and formally integrate by parts to get
\[
\int_{\Omega_1}\frac{g\left(| D  u_1|\right)}{| D  u_1|} D  u_1\cdot D\varphi\,{\rm d}x=-\int_{\Gamma}\left(\frac{g\left(| D  u_1|\right)}{| D  u_1|} D  u_1\cdot \nu\right)\varphi\,{\rm d}\mathcal{H}^{d-1}
\]
and
\[
\int_{\Omega_2}\frac{g\left(| D  u_2|\right)}{| D  u_2|} D  u_2\cdot D\varphi\,{\rm d}x=-\int_{\Gamma}\left(\frac{g\left(| D  u_2|\right)}{| D  u_2|} D  u_2\cdot \nu\right)\varphi\,{\rm d}\mathcal{H}^{d-1}
\]
Adding and using \eqref{eq_stima119}, we obtain
\[
\int_{\Omega}\frac{g\left(| D  u|\right)}{| D  u|} D  u\cdot D  \varphi\,{\rm d}x=\int_{\Gamma}f \varphi\,{\rm d}\mathcal{H}^{d-1},\hspace{.2in}\forall\,\varphi\in W^{1,G}(\Omega).
\]

\begin{Remark}
We notice the integrals in Definition \ref{def_weaksol} are well-defined. Indeed, let $u,v\in W^{1,G}(\Omega)$; we verify that
\begin{equation*}
	\int_{\Omega}\frac{g(|Du|)}{|Du|}Du\cdot Dv\,{\rm d}x<\infty.
\end{equation*}
Since $g$ is increasing, we have $tg(t)\le CG(t)$ for $t\ge0$. Also, $G(t+s)\le C\big(G(t)+G(s)\big)$ for $t,s\ge0$. Hence,
\begin{align*}
\bigg|\int_{\Omega}\frac{g(|Du|)}{|Du|}Du\cdot Dv\,{\rm d}x\bigg|\le&\int_{\Omega}g(|Du|)|Dv|\,{\rm d}x\\
\le&\int_{\Omega}g(|Du|+|Dv|)(|Du|+|Dv|)\,{\rm d}x\\
\le&C\int_{\Omega}G(|Du|+|Dv|)\,{\rm d}x\\
\le&C\int_{\Omega}G(|Du|)\,{\rm d}x+C\int_{\Omega}G(|Dv|)\,{\rm d}x\\
<&\infty.
\end{align*}
\end{Remark}

\begin{Remark}
Let $u\in W_0^{1,G}(\Omega)$, and suppose that \eqref{monicavitti} is in force. Then one infers $u\in W_0^{1,p}(\Omega)$. Indeed, that inequality yields
\begin{align*}
\int_\Omega|Du|^p\,{\rm d}x\le&C\int_\Omega g(|Du|)|Du|\,{\rm d}x\\
\le&C\int_\Omega G(|Du|)\,{\rm d}x.
\end{align*}
\end{Remark}

\subsection{Existence and uniqueness of weak solutions}
To prove the existence of a unique weak solution to \eqref{eq_stima118}-\eqref{eq_stima119}, one can resort to approximation and monotonicity methods. We refer the reader to \cite{baroni2015}; see also \cite{Lieberman_1991}. Additionally, we remark that the weak solution is the global minimiser of the functional $I: W_0^{1,G}(\Omega)\to\mathbb{R}$ defined by
\begin{equation}\label{stima117}
 I(u)=\int_{\Omega}G\left(| D u|\right)\,{\rm d}x-\int_{\Gamma}fu\,{\rm d}\mathcal{H}^{d-1},
\end{equation}
whose Euler-Lagrange equation, in its weak formulation, is precisely \eqref{portia}.

\subsection{Auxiliary results}\label{subsec_prelim} We now collect some auxiliary material which will be instrumental in the proofs of the main results. We start with a technical inequality (c.f. \cite[Lemma 2]{Serrin_1964}).

\begin{Lemma}\label{lemma numerico}
Let $p>0$, and $N\in\mathbb{N}$. Let also $a_1,\dots,a_N,q_1,\dots,q_N$ be real numbers such that $0<a_i<\infty$ and $0\le q_i<p$, for every $i=1,\ldots,N$. Suppose that $z,z_1,\ldots,z_N$ are positive real numbers satisfying
\begin{equation*}
z^p\le\sum_{i=1}^Na_iz_i^{q_i}.
\end{equation*}
Then there exists $C>0$ such that
\begin{equation*}
z\le C\sum_{i=1}^Na_i^{\gamma_i}
\end{equation*}
where $\gamma_i=(p-q_i)^{-1}$, for $i=1,\ldots,N$. Finally, $C=C(N,p,q_1,\ldots,q_N)$.
\end{Lemma}

Although standard in the field, the following result lacks detailed proof in the literature. We include it here for completeness and future reference.

\begin{Lemma}\label{lem_stima119}
Fix $R_0>0$ and let $\phi:[0,R_0]\to[0,\infty)$ be a non-decreasing function. Suppose there exist constants $C_1,\alpha,\beta>0$, and $C_2,\mu\ge0$, with $\beta<\alpha$, satisfying
\begin{equation*}
\phi(r)\le C_1\Big[\Big(\frac{r}{R}\Big)^{\alpha}+\mu\Big]\phi(R)+C_2R^{\beta},
\end{equation*}
for every $0<r\le R\le R_0$.
Then, for every $\sigma\le\beta$, there exists $\mu_0=\mu_0(C_1,\alpha,\beta,\sigma)$ such that, if $\mu<\mu_0$, for every $0<r\le R\le R_0$, we have
\begin{equation*}
\phi(r)\le C_3\Big(\frac{r}{R}\Big)^{\sigma}\big(\phi(R)+C_2R^{\sigma}\big),
\end{equation*}
where $C_3=C_3(C_1,\alpha,\beta,\sigma)>0$. Moreover,
\begin{equation*}
\phi(r)\le C_4r^{\sigma},
\end{equation*}
where $C_4=C_4(C_2,C_3,R_0,\phi(R_0),\sigma)$.
\end{Lemma}
\begin{proof}
For clarity, we split the proof into two steps. First, an induction argument leads to an inequality at discrete scales; then, we pass to the continuous case and conclude the argument.  

\medskip

\noindent{\bf Step 1 -} We want to verify that 
\begin{equation}\label{eq_1}
\phi(\theta^{n+1}R)\le\theta^{(n+1)\delta}\phi(R)+C_2\theta^{n\beta}R^{\beta}\sum_{j=0}^n\theta^{j(\delta-\beta)},
\end{equation}
for every $n\in\mathbb{N}$. We notice it suffices to prove the estimate for $\sigma=\beta$ and work in this setting. For $0<\theta<1$ and $0<R\le R_0$ the assumption of the lemma yields
\[
\phi(\theta R)\le C_1\bigg[\bigg(\frac{\theta R}{R}\bigg)^{\alpha}+\mu\bigg]\phi(R)+C_2R^{\beta}=\theta^{\alpha}C_1(1+\mu\theta^{-\alpha})\phi(R)+C_2R^{\beta}.
\]
Choose $\theta\in(0,1)$ such that $2C_1\theta^{\alpha}=\theta^{\delta}$ with $\beta<\delta<\alpha$. Notice that $\theta$ depends only on $C_1,\alpha,\delta$. Take $\mu_0>0$ such that $\mu_0\theta^{-\alpha}<1$. For every $R\le R_0$ we then have
\begin{equation}\label{eq_induction00}
\phi(\theta R)\le\theta^{\delta}\phi(R)+C_2R^{\beta}
\end{equation}
and the base case follows. Suppose the statement has already been verified for some $k\in\mathbb{N}$, $k\ge2$; then
\begin{equation*}
\phi(\theta^kR)\le\theta^{k\delta}\phi(R)+C_2\theta^{(k-1)\beta}R^{\beta}\sum_{j=0}^{k-1}\theta^{j(\delta-\beta)}.
\end{equation*}
Thanks to \eqref{eq_induction00}, we have
\[
\begin{split}
\phi(\theta^{k+1}R)&=\phi\big(\theta^k(\theta R)\big)\le\theta^{k\delta}\phi(\theta R)+C_2\theta^{(k-1)\beta}(\theta R)^{\beta}\sum_{j=0}^{k-1}\theta^{j(\delta-\beta)}\\
&\le\theta^{k\delta}\big[\theta^{\delta}\phi(R)+C_2R^{\beta}\big]+C_2\theta^{k\beta}R^{\beta}\sum_{j=0}^{k-1}\theta^{j(\delta-\beta)}\\
&=\theta^{(k+1)\delta}\phi(R)+C_2\theta^{k\delta}R^{\beta}+C_2\theta^{k\beta}R^{\beta}\sum_{j=0}^{k-1}\theta^{j(\delta-\beta)}\\
&=\theta^{(k+1)\delta}\phi(R)+C_2\theta^{k\beta}R^{\beta}\sum_{j=0}^k\theta^{j(\delta-\beta)}.
\end{split}
\]
Hence, \eqref{eq_1} holds for every $k\in\mathbb{N}$, and the induction argument is complete.

\medskip

\noindent{\bf Step 2 -} Next, we pass from the discrete to the continuous case. In particular, we claim that
\begin{equation*}
\phi(r)\le C_3\Big(\frac{r}{R}\Big)^{\beta}\big(\phi(R)+C_2R^{\beta}\big),
\end{equation*}
for every $0<r\le R\le R_0$.

Indeed,
\begin{align*}
\phi(\theta^{k+1}R) \le&\theta^{(k+1)\delta}\phi(R)+C_2\theta^{k\beta}R^{\beta}\frac{1}{1-\theta^{\delta-\beta}}\\
=&\theta^{(k+1)\delta}\phi(R)+C_2R^{\beta}\frac{\theta^{(k+1)\beta}}{\theta^{\beta}-\theta^{\delta}}\\
\le&C_3\theta^{(k+1)\beta}\big(\phi(R)+C_2R^{\beta}\big),
\end{align*}
for every $k\in\mathbb{N}$. Taking $k\in\mathbb{N}$ such that $\theta^{k+2}R\le r<\theta^{k+1}R$, up to relabeling the constant $C_3$, we get
\begin{align*}
\phi(r)\le&\phi(\theta^{k+1}R)\le C_3\theta^{(k+1)\beta}\big(\phi(R)+C_2R^{\beta}\big)\\
=&C_3\theta^{(k+2)\beta}\theta^{-\beta}\big(\phi(R)+C_2R^{\beta}\big)\\
\le&C_3\Big(\frac{r}{R}\Big)^{\beta}\big(\phi(R)+C_2R^{\beta}\big).
\end{align*}
Finally, one notices 
\[
\phi(r)\le C_3\frac{1}{R_0^{\beta}}\big(\phi(R_0)+C_2R_0^{\beta}\big)r^{\beta}=:C_4r^{\beta},
\]
and the proof is complete.
\end{proof}

We conclude this section by introducing two functional spaces we resort to in the paper, namely Campanato and Morrey spaces. Indeed, we use embedding properties of these spaces to conclude the H\"older-continuity of weak solutions when the interface data is unbounded.

\begin{Definition}[Campanato spaces]
We denote by $L_C^{p,\lambda}(\Omega;\mathbb{R}^d)$, with $1\le p<\infty$ and $\lambda\ge0$, the space of functions $u\in L^p(\Omega;\mathbb{R}^d)$ such that
\begin{equation*}
[u]_{L_C^{p,\lambda}(\Omega;\mathbb{R}^d)}^p=\sup_{x^0\in\Omega,\rho>0}\frac{1}{\rho^{\lambda}}\int_{\Omega\cap B(x^0,\rho)}|u-(u)_{\Omega\cap B(x^0,\rho)}|^p\,{\rm d}x<\infty.
\end{equation*}
\end{Definition}

\begin{Definition}[Morrey spaces]
We denote by $L_M^{p,\lambda}(\Omega;\mathbb{R}^d)$, with $1\le p<\infty$ and $\lambda\ge0$, the space of functions $u\in L^p(\Omega;\mathbb{R}^d)$ such that
\begin{equation*}
\|u\|_{L_M^{p,\lambda}(\Omega)}^p=\sup_{x^0\in\Omega,\rho>0}\frac{1}{\rho^{\lambda}}\int_{\Omega\cap B(x^0,\rho)}|u|^p\,{\rm d}x<\infty.
\end{equation*}
\end{Definition}

Notice that $L^{p,\lambda}_{M}$ and $L^{p,\lambda}_{C}$ are isomorphic; see \cite[Proposition 2.3]{giusti}. We recall that a function $u\in W^{1,1}(\Omega)$ such that $Du\in L_M^{p,\lambda}(\Omega;\mathbb{R}^d)$ is H\"older continuous. More precisely, we have $u\in C^{0,\alpha}(\Omega)$ with $\alpha=1-\lambda/p$; see \cite{adamsmorrey}.

\section{Local boundedness}\label{sec_alkhawarizmi}

In this section, we prove the local boundedness for the weak solutions to a particular variant of our problem. Namely, we consider the case $g(t):=t^{p-1}$ and recover the $p-$Laplace operator. Our argument is inspired by the one put forward in \cite{Serrin_1964}. 

\begin{Theorem}[Local Boundedness]\label{thm_lb}
Let $u\in W_0^{1,p}(\Omega)$ be the weak solution to \eqref{eq_stima118}-\eqref{eq_stima119}, with $g(t):=t^{p-1}$ and $f \in L^{\infty}(\Gamma)$. Then for any $B_{R}:=B_{R}(x_0)\Subset \Omega$, there exists $C=C\big(d,p,R,\|g\|_{L^{\infty}(\Gamma)}\big)>0$ such that
\[
	\|u\|_{L^{\infty}(B_{R/2})}\le CR^{-\frac{d}{p}}\big(\|u\|_{L^p(B_R)}+R^{\frac{d}{p}+1}\|f\|_{L^{\infty}(\Gamma)}\big) 
\]
and
\[
	\| D  u\|_{L^p(B_{R/2})}\le CR^{-1}\big(\|u\|_{L^p(B_R)}+R^{\frac{d}{p}+1}\|f\|_{L^{\infty}(\Gamma)}\big).
\]
\end{Theorem}
\begin{proof}
Fix $R>0$ such that $B_R\Subset \Omega$ and set $k:=R\|f\|_{L^{\infty}(\Gamma)}$. Define $\overline{u}:\Omega\to\mathbb{R}$ as
\[
	\overline{u}(x):=|u(x)|+k
\]
for all $x\in\Omega$. Fix $q\ge1$ and $\ell>k$. For $t\in\mathbb{R}$, denote $\overline{t}:=|t|+k$. To ease the presentation, we split the remainder of the proof into four steps.

\medskip

\noindent{\bf Step 1 -} Let $F:[k,\infty)\to\mathbb{R}$ be defined as 
\[
F(s):=
\begin{cases}
	s^q&\hspace{.3in}\mbox{if}\hspace{.3in}k\le s\le \ell\\
	q\ell^{q-1}s-(q-1)\ell^q&\hspace{.3in}\mbox{if}\hspace{.3in}\ell<s.
\end{cases}
\]
Then $F\in C^1\big([k,\infty)\big)$ and $F\in C^\infty\big([k,\infty)\setminus{\{\ell\}}\big)$. Let $H:\mathbb{R}\to\mathbb{R}$ be defined as
\begin{equation*}
	H(t):=\textnormal{sgn}(t)\big(F(\overline{t})F'(\overline{t})^{p-1}-q^{p-1}k^{\beta}\big),\quad\forall t\in\mathbb{R},
\end{equation*}
where $\beta=p(q-1)+1>1$. A simple computation yields
\begin{align*}
H'(t)=
\begin{cases}
	q^{-1}\beta F'(\overline{t})^p&\hspace{.3in}\mbox{if}\hspace{.3in}|t|<\ell-k\\
	F'(\overline{t})^p&\hspace{.3in}\mbox{if}\hspace{.3in}|t|>\ell-k.
\end{cases}
\end{align*}
Notice that
\[
	|H(u)|\le F(\overline{u})F'(\overline{u})^{p-1}
\]
and 
\[
	\overline{u}F'(\overline{u})\le qF(\overline{u}).	
\]

\noindent{\bf Step 2 -} In this step, we introduce auxiliary test functions, which build upon the former inequalities. Fix $0<r<R$. Let $\eta\in C_c^{\infty}(B_R)$, $0\le\eta\le1$, $\eta=1$ in $B_r$, $| D \eta|\le(R-r)^{-1}$. Let $v=\eta^pG(u)$. Since $G\in C^1\big(\mathbb{R}\setminus\{\pm(\ell-k)\}\big)$ is continuous, with bounded derivative, it follows that $G(u)\in W^{1,p}(\Omega)$. Hence $v$ is an admissible test function. We have
\[
	 D  v=
		\begin{cases}
			p\eta^{p-1}H(u) D \eta+\eta^pH'(u) D  u&\hspace{.2in}\mbox{if}\hspace{.2in}u\ne\pm(\ell-k)\\
			p\eta^{p-1}H(u) D \eta&\hspace{.2in}\mbox{if}\hspace{.2in}u=\pm(\ell-k).
		\end{cases}
\]
				
Set $w(x)=F\big(\overline{u}(x)\big)$. Notice that $q^{-1}\beta\ge1$; hence $H'(u)\le q^{-1}\beta F'(\overline{u})^p$. Notice also that $| D  u|=| D \overline{u}|$.
		
Using the trace theorem and the Poincar\'e inequality, we get
\begin{equation}\label{stima2}
	\int_{\Omega}| D  u|^{p-2} D  u\cdot D  v\,{\rm d}x\le\|f\|_{L^{\infty}(\Gamma)}\int_{\Gamma}|v|\,{\rm d}\mathcal{H}^{d-1}\le C\int_{\Omega}| D  v|\,{\rm d}x.
\end{equation}

Now we estimate the left-hand side of \eqref{stima2} from below. We get
\begin{align}\label{stima5}\notag
	\int_{B_1}| D  u|^{p-2} D  u\cdot D  v\,{\rm d}x=&\int_{B_1}| D  u|^{p-2} D  u\cdot\big(p\eta^{p-1}H(u) D \eta+\eta^pH'(u) D  u\big)\,{\rm d}x \notag\\
		=&p\int_{B_1}\eta^{p-1}H(u)| D  u|^{p-2} D  u\cdot D \eta\,{\rm d}x \notag\\
			&+\int_{B_1}\eta^pH'(u)| D  u|^p\,{\rm d}x \notag\\
		\ge&-p\int_{B_1}\eta^{p-1}F(\overline{u})F'(\overline{u})^{p-1}| D \overline{u}|^{p-1}| D \eta|\,{\rm d}x \notag\\
			&+\int_{B_1}\eta^pF'(\overline{u})^p| D \overline{u}|^p\,{\rm d}x \notag\\
		=&-p\int_{B_1}\eta^{p-1}w| D  w|^{p-1}| D \eta|\,{\rm d}x \notag\\
			&+\int_{B_1}\eta^p| D  w|^p\,{\rm d}x \notag\\
		\ge&-p\|w D \eta\|_{L^p(B_1)}\|\eta D  w\|_{L^p(B_1)}^{p-1}+\|\eta D  w\|_{L^p(B_1)}^p.
\end{align}
We also control the right-hand side of (\ref{stima2}) by computing
\begin{align}\label{stima3}\notag
	C\int_{B_1}| D  v|\,{\rm d}x=&C\int_{B_1}\frac{\overline{u}^{p-1}}{\overline{u}^{p-1}}|p\eta^{p-1}H(u) D \eta+\eta^pH'(u) D  u|\,{\rm d}x \notag\\
		\le&Ck^{1-p}p\int_{B_1}\overline{u}^{p-1}\eta^{p-1}|H(u) D \eta|\,{\rm d}x \notag\\
			&+Ck^{1-p}\int_{B_1}\overline{u}^{p-1}\eta^pH'(u)| D  u|\,{\rm d}x \notag\\
		\le&C\int_{B_1}\overline{u}^{p-1}\eta^{p-1}F(\overline{u})F'(\overline{u})^{p-1}| D \eta|\,{\rm d}x \notag\\
			&+Cq^{-1}\beta\int_{B_1}\overline{u}^{p-1}\eta^pF'(\overline{u})^p| D  u|\,{\rm d}x \notag\\
		\le&C\int_{B_1}\eta^{p-1}q^{p-1}F(\overline{u})^{p-1}F(\overline{u})| D \eta|\,{\rm d}x \notag\\
			&+Cq^{-1}\beta\int_{B_1}q^{p-1}F(\overline{u})^{p-1}\eta^p F'(\overline{u})| D  u|\,{\rm d}x \notag\\
		=&Cq^{p-1}\int_{B_1}(\eta w)^{p-1}w| D \eta|\,{\rm d}x \notag\\
			&+Cq^{p-2}\beta\int_{B_1}(\eta w)^{p-1}\eta| D  w|\,{\rm d}x \notag\\
		\le&Cq^{p-1}\|\eta w\|_{L^p(B_1)}^{p-1}\|w D \eta\|_{L^p(B_1)} \notag\\
			&+Cq^{p-2}\beta\|\eta w\|_{L^p(B_1)}^{p-1}\|\eta D  w\|_{L^p(B_1)}.
\end{align}
From (\ref{stima2}), combining (\ref{stima3}) with (\ref{stima5}), we get
\begin{align}\label{stima6}
	\|\eta D  w\|_{L^p(\Omega)}^p\le&p\|w D \eta\|_{L^p(\Omega)}\|\eta D  w\|_{L^p(\Omega)}^{p-1} \notag\\
		&+Cq^{p-1}\|\eta w\|_{L^p(\Omega)}^{p-1}\|w D \eta\|_{L^p(\Omega)} \notag\\
		&+Cq^{p-1}\|\eta w\|_{L^p(\Omega)}^{p-1}\|\eta D  w\|_{L^p(\Omega)},
\end{align}
where we have used 
\[
	\beta=pq-p+1\le pq-p+q\le pq+q=(p+1)q.
\] 

\noindent{\bf Step 3 -} Set
\begin{equation}
z=\frac{\|\eta D  w\|_{L^p(\Omega)}}{\|w D \eta\|_{L^p(\Omega)}},\quad\zeta=\frac{\|\eta w\|_{L^p(\Omega)}}{\|w D \eta\|_{L^p(\Omega)}}. \notag
\end{equation}
By dividing (\ref{stima6}) for $\|w D \eta\|_{L^p(\Omega)}^p$, we have
\begin{align}
z^p\le&pz^{p-1}+Cq^{p-1}\frac{\|\eta w\|_{L^p(\Omega)}^{p-1}}{\|w D \eta\|_{L^p(\Omega)}^{p-1}}+Cq^{p-1}\frac{\|\eta w\|_{L^p(\Omega)}^{p-1}}{\|w D \eta\|_{L^p(\Omega)}^{p-1}}\frac{\|\eta D  w\|_{L^p(\Omega)}}{\|w D \eta\|_{L^p(\Omega)}} \notag\\
=&pz^{p-1}+Cq^{p-1}\zeta^{p-1}+Cq^{p-1}\zeta^{p-1}z. \notag
\end{align}
An application of Lemma \ref{lemma numerico}, implies
\begin{equation*}
	z\le C\big(p+q^{\frac{p-1}{p}}\zeta^{\frac{p-1}{p}}+q\zeta\big)\le Cq(1+\zeta),
\end{equation*}
giving
\begin{equation}\label{stima7}
\|\eta D  w\|_{L^p(\Omega)}\le Cq\big(\|\eta w\|_{L^p(\Omega)}+\|w D \eta\|_{L^p(\Omega)}\big).
\end{equation}
Using the Sobolev inequality, we get
\begin{align}
\|\eta w\|_{L^{p^*}(\Omega)}\le&C\| D (\eta w)\|_{L^p(\Omega)} \notag\\
\le&C\big(\|w D \eta\|_{L^p(\Omega)}+\|\eta D  w\|_{L^p(\Omega)}\big) \notag\\
\le&C\Big[\|w D \eta\|_{L^p(\Omega)}+Cq\big(\|\eta w\|_{L^p(\Omega)}+\|w D \eta\|_{L^p(\Omega)}\big)\Big] \notag
\end{align}
and so
\begin{equation}\label{stima8}
\|\eta w\|_{L^{p^*}(\Omega)}\le Cq\big(\|\eta w\|_{L^p(\Omega)}+\|w D \eta\|_{L^p(\Omega)}\big).
\end{equation}
Recall that $\eta=1$ in $B_r$ and $| D \eta|\le(R-r)^{-1}$. Hence, (\ref{stima7}) becomes
\begin{equation}\label{stima11}
	\begin{split}
\| D  w\|_{L^p(B_r)}\le& Cq\Bigg[\bigg(\int_{B_R}w^p\,{\rm d}x\bigg)^{\frac{1}{p}}+\frac{1}{R-r}\bigg(\int_{B_R}w^p\,{\rm d}x\bigg)^{\frac{1}{
p}}\Bigg] \notag\\
=&Cq\|w\|_{L^p(B_R)}\bigg(1+\frac{1}{R-r}\bigg) \notag\\
=&Cq\frac{R-r+1}{R-r}\|w\|_{L^p(B_R)} \notag\\
\le&Cq\frac{\textnormal{diam}(B_1)+1}{R-r}\|w\|_{L^p(B_R)} \notag\\
\le&Cq\frac{1}{R-r}\|w\|_{L^p(B_R)}.
	\end{split}
\end{equation}
Similarly, (\ref{stima8}) becomes
\begin{equation}\label{stima9}
\|w\|_{L^{p^*}(B_r)}\le Cq\frac{1}{R-r}\|w\|_{L^p(B_R)}.
\end{equation}
We claim that $F_\ell\le F_{\ell+1}$, for every $\ell\in\mathbb{N}$, $\ell>k$. The only non-trivial case is when $\ell<\overline{t}\le \ell+1$. In this case, we have 
$$F_\ell(\overline{t})=q\ell^{q-1}\overline{t}-(q-1)\ell^q$$ 
and 
$$F_{\ell+1}(\overline{t})=\overline{t}^q.$$
Let $h:(\ell,\ell+1]\to\mathbb{R}$ be defined by
$$h(\overline{t})=\overline{t}^q-q\ell^{q-1}\overline{t}+(q-1)\ell^q.$$
We have $h'(\overline{t})=q\overline{t}^{q-1}-q\ell^{q-1}>0$, for every $\overline{t}\in(\ell,\ell+1]$, and hence $h$ is an increasing function. Since $\lim_{\overline{t}\to \ell}h(\overline{t})=0$, we have $h\ge0$ in $(\ell,\ell+1]$, and so $F_\ell\le F_{\ell+1}$. Letting $\ell\to\infty$ in (\ref{stima9}), since $0\le F_\ell\le F_{\ell+1}$ for every $\ell\in\mathbb{N}$, $\ell>k$, by the Monotone Convergence Theorem, we obtain
\begin{equation}
\bigg(\int_{B_r}\overline{u}^{qp^*}\,{\rm d}x\bigg)^{\frac{1}{p^*}}\le Cq\frac{1}{R-r}\bigg(\int_{B_R}\overline{u}^{qp}\,{\rm d}x\bigg)^{\frac{1}{p}}. \notag
\end{equation}
Set 
\[
	s:=qp\hspace{.3in}\mbox{and}\hspace{.3in}\gamma:=p^*/p=d/(d-p);
\] 
then
\begin{equation*}
	\bigg(\int_{B_r}\overline{u}^{s\gamma}\,{\rm d}x\bigg)^{\frac{1}{p\gamma}}\le Cq\frac{1}{R-r}\bigg(\int_{B_R}\overline{u}^{s}\,{\rm d}x\bigg)^{\frac{1}{p}}.
\end{equation*}
Raising both sides of the previous inequality to $p/s$, one gets
\begin{equation}\label{stima10}
\bigg(\int_{B_r}\overline{u}^{s\gamma}\,{\rm d}x\bigg)^{\frac{1}{s\gamma}}\le C^{\frac{p}{s}}\bigg(\frac{s}{p}\bigg)^{\frac{p}{s}}\Big(\frac{1}{R-r}\Big)^{\frac{p}{s}}\bigg(\int_{B_R}\overline{u}^{s}\,{\rm d}x\bigg)^{\frac{1}{s}}.
\end{equation}
Set $s_j=s\gamma^j$ and $r_j=r+2^{-j}(R-r)$, for every $j\in\mathbb{N}_0$. Iterating (\ref{stima10}), which holds for every $s\ge p$, we have
\begin{align*}
\bigg(\int_{B_{r_{j+1}}}\overline{u}^{s_j\gamma}\,{\rm d}x\bigg)^{\frac{1}{s_j\gamma}}\le&C^{\frac{p}{s_j}}\bigg(\frac{s_j}{p}\bigg)^{\frac{p}{s_j}}2^{\frac{p}{s_j}(j+1)}\Big(\frac{1}{R-r}\Big)^{\frac{p}{s_j}}\bigg(\int_{B_{r_j}}\overline{u}^{s_j}\,{\rm d}x\bigg)^{\frac{1}{s_j}} \notag\\
=&C^{\frac{p}{s_{j-1}\gamma}}\bigg(\frac{s_{j-1}\gamma}{p}\bigg)^{\frac{p}{s_{j-1}\gamma}}2^{\frac{p}{s_{j-1}\gamma}(j+1)}\Big(\frac{1}{R-r}\Big)^{\frac{p}{s_{j-1}\gamma}} \notag\\
&\times\bigg(\int_{B_{r_j}}\overline{u}^{s_{j-1}\gamma}\,{\rm d}x\bigg)^{\frac{1}{s_{j-1}\gamma}} \notag\\
\le& C(j,p,s,d)\Big(\frac{1}{R-r}\Big)^{\frac{p}{s}\sum_{k=0}^j\gamma^{-k}}\bigg(\int_{B_R}\overline{u}^s\,{\rm d}x\bigg)^{\frac{1}{s}},
\end{align*}
where
\[
	C(j,p,s,d):=C^{\frac{p}{s}\sum_{k=0}^j\gamma^{-k}}\bigg(\frac{s}{p}\bigg)^{\frac{p}{s}\sum_{k=0}^j{\gamma^{-k}}}\gamma^{\frac{p}{s}\sum_{k=0}^jk\gamma^{-k}}2^{\frac{p}{s}\sum_{k=0}^j(k+1)\gamma^{-k}}.
\]
Notice that $r<r_{j}$, for every $j\in\mathbb{N}_0$, the series are convergent and in particular $\sum_{k=0}^\infty\gamma^{-k}=d/p$. By letting $j\to\infty$, we get
\begin{equation}
\sup_{B_r}\overline{u}\le C\bigg(\frac{1}{(R-r)^d}\int_{B_R}\overline{u}^s\,{\rm d}x\bigg)^{\frac{1}{s}}.
\end{equation}

\noindent{\bf Step 4 - }Now, we can choose some parameters in the former inequalities to complete the proof. By choosing $q=1$, setting $r:=R/2$, and recalling that $\overline{u}=|u|+k$, we get
\begin{align}
\|u\|_{L^{\infty}(B_{R/2})}\le&\|\overline{u}\|_{L^{\infty}(B_{R/2})}\le CR^{-\frac{d}{p}}\big(\|u\|_{L^p(B_R)}+R^{\frac{d}{p}}k\big). \notag
\end{align}
The second inequality in the theorem follows by setting $q=1$ and $r:=R/2$ in \eqref{stima11}, obtaining
\begin{align}
\| D  u\|_{L^p(B_{R/2})}=&\| D \overline{u}\|_{L^p(B_{R/2})}\\
\le & CR^{-1}\|\overline{u}\|_{L^p(B_R)} \notag\\
\le&CR^{-1}\big(\|u\|_{L^p(B_R)}+\|k\|_{L^p(B_R)}\big) \notag\\
\le&CR^{-1}\big(\|u\|_{L^p(B_R)}+R^{\frac{d}{p}}k\big). \notag
\end{align}
\end{proof}

\section{Gradient regularity estimates in ${\rm BMO}-$spaces}\label{sec_beacon}

In this section, we prove regularity estimates for weak solutions. In case $f\in L^\infty(\Gamma)$, we prove that $Du\in {\rm BMO}_{\rm loc}(\Omega)$. In addition, we allow the interface data to be unbounded, provided it belongs to a Sobolev space $W^{1,p'+\varepsilon}(\Omega)$, where the parameter $\varepsilon>0$ is to be set further. In this case, we verify that $u\in C^{0,\alpha}_{\rm loc}(\Omega)$, for some $\alpha\in(0,1)$ depending only on the dimension, $p$ and $\varepsilon$.

We proceed with an auxiliary lemma. For $w\in W^{1,G}(\Omega)$, let  $W^{1,G}_w(\Omega)$ denote the Orlicz-Sobolev space comprising the functions $u\in W^{1,G}(\Omega)$ such that 
\[
	u-w\in W^{1,G}_0(\Omega).
\]

\begin{Lemma}\label{stima146}
		Let $w\in W^{1,G}(B_R)$. Suppose \eqref{business class}--\eqref{monicavitti} is in force. Suppose further $h\in W_w^{1,G}(B_R)$ is a weak solution to
		\begin{equation*}
			\textnormal{div}\bigg(\frac{g(| D  h|)}{| D  h|} D  h\bigg)=0\quad\textnormal{in }B_R.
		\end{equation*}
Then there exists $C>0$ such that
		\begin{equation}\label{antonioni}
			\int_{B_{R}}G(| D  w|)-G(| D  h|)\,{\rm d}x\ge C\int_{B_{R}}| D (w-h)|^p\,{\rm d}x.
		\end{equation}
	\end{Lemma}
	\begin{proof}$\\$
			Let $\tau\in[0,1]$, define $v_\tau=\tau w+(1-\tau)h$. The monotonicity condition in \eqref{monicavitti} implies
			\begin{align*}
				\int_{B_{R}}G(| D  w|)-G(| D  h|)\,{\rm d}x=&\int_0^1\frac{d}{d\tau}\bigg(\int_{B_{R}}G(| D  v_{\tau}|)\,{\rm d}x\bigg)\,{\rm d}\tau\\
				=&\int_0^1\int_{B_{R}}\frac{d}{d\tau}G(| D  v_\tau|)\,{\rm d}x\,{\rm d}\tau\\
				=&\int_0^1\int_{B_{R}}\frac{g(| D  v_\tau|)}{| D  v_\tau|} D  v_\tau\cdot D (w-h)\,{\rm d}x\,{\rm d}\tau\\
				=&\int_0^1\frac{1}{\tau}\int_{B_{R}}\bigg(\frac{g(| D  v_\tau|)}{| D  v_\tau|} D  v_\tau-\frac{g(| D  h|)}{| D  h|} D  h\bigg)\\
				&\cdot D (v_\tau-h)\,{\rm d}x\,{\rm d}\tau\\
				\ge&C\int_0^1\frac{1}{\tau}\int_{B_{R}}| D (v_\tau-h)|^p\,{\rm d}x\,{\rm d}\tau\\
				=&C\int_{B_{R}}| D (w-h)|^p\,{\rm d}x,
			\end{align*}
			and the proof is complete.
		\end{proof}

%\begin{Remark}\normalfont
%We use \eqref{antonioni} for different values of $p$, according to the integrability of the interface data.
%\end{Remark}

\subsection{Regularity estimates in ${\rm BMO}-$spaces}

In this section, we suppose $f\in L^\infty(\Omega)$ and establish ${\rm BMO}-$regularity estimates for the gradient of solutions. We start by recalling a proposition from \cite{baroni2015}.

\begin{Proposition}\label{stima136}
Let $h\in W^{1,G}(B_R)$ be a weak solution of
	\begin{equation*}
		\textnormal{div}\bigg(\frac{g(| D  h|)}{| D  h|} D  h\bigg)=0\quad\textnormal{in }B_R.
	\end{equation*}
Suppose \eqref{business class}--\eqref{monicavitti} are in force. Then there exist $C>0$ and $\alpha\in(0,1)$ such that, for every $r\in(0,R]$, we have
	\begin{equation*}
		\int_{B_r}| D  h-( D  h)_r|\,{\rm d}x\le C\Big(\frac{r}{R}\Big)^{d+\alpha}\int_{B_R}| D  h-( D  h)_R|\,{\rm d}x.
	\end{equation*}
\end{Proposition}

For a proof of Proposition \ref{stima136}, we refer the reader to \cite{baroni2015}.

\begin{Proposition}\label{stima161}
		Let $w\in W^{1,G}(B_R)$, and suppose $h\in W^{1,G}(B_R)$ is a weak solution of
		\begin{equation*}
			\textnormal{div}\bigg(\frac{g(| D  h|)}{| D  h|} D  h\bigg)=0\quad\textnormal{in }B_R.
		\end{equation*}
Suppose \eqref{business class}--\eqref{monicavitti} are in force. Then there exists $C>0$ such that, for every $0<r\le R$, we have
		\begin{align}
			\int_{B_r}| D  w-( D  w)_r|\,{\rm d}x\le& C\Big(\frac{r}{R}\Big)^{d+\alpha}\int_{B_R}| D  w-( D  w)_R|\,{\rm d}x \notag\\
			&+C\int_{B_R}| D  w- D  h|\,{\rm d}x, \notag
		\end{align}
		where $\alpha$ is given by Proposition \ref{stima136}.
	\end{Proposition}
\begin{proof}Let $r\in(0,R]$. We have
			\begin{align}\label{stima137}
				\int_{B_r}| D  w-( D  w)_r|\,{\rm d}x\le&\int_{B_r}| D  w-( D  h)_r|\,{\rm d}x \notag\\
				&+\int_{B_r}|( D  w)_r-( D  h)_r|\,{\rm d}x.
			\end{align}
			Similarly, we have
			\begin{align}\label{stima138}
				\int_{B_r}| D  w-( D  h)_r|\,{\rm d}x\le&\int_{B_r}| D  w- D  h|\,{\rm d}x \notag\\
				&+\int_{B_r}| D  h-( D  h)_r|\,{\rm d}x.
			\end{align}
			Moreover,
			\begin{align}\label{stima139}
				\int_{B_r}|( D  w)_r-( D  h)_r|\,{\rm d}x=&|( D  w)_r-( D  h)_r|\int_{B_r}\,{\rm d}x \notag\\
				=&|B_r|\bigg|\frac{1}{|B_r|}\int_{B_r} D  w- D  h\,{\rm d}x\bigg| \notag\\
				\le&\int_{B_r}| D  w- D  h|\,{\rm d}x.
			\end{align}
			Combining \eqref{stima137}, \eqref{stima138} with \eqref{stima139}, we get
			\begin{align}\label{stima142}
				\int_{B_r}| D  w-( D  w)_r|\,{\rm d}x\le&\int_{B_r}| D  h-( D  h)_r|\,{\rm d}x \notag\\
				&+2\int_{B_r}| D  w- D  h|\,{\rm d}x.
			\end{align}
			Changing the roles of $w$ and $h$ and integrating in the ball $B_R$, we obtsain
			\begin{align}\label{stima140}
				\int_{B_R}| D  h-( D  h)_R|\,{\rm d}x\le&\int_{B_R}| D  w-( D  w)_R|\,{\rm d}x \notag\\
				&+2\int_{B_R}| D  w- D  h|\,{\rm d}x.
			\end{align}
			Thanks to Proposition \ref{stima136}, we conclude
			\begin{align}\label{stima141}
				\int_{B_r}| D  w-( D  w)_r|\,{\rm d}x\le&C\Big(\frac{r}{R}\Big)^{d+\alpha}\int_{B_R}| D  h-( D  h)_R|\,{\rm d}x \notag\\
				&+C\int_{B_R}| D  w- D  h|\,{\rm d}x.
			\end{align}
			Combining \eqref{stima142}, \eqref{stima140} with \eqref{stima141} we get
			\begin{align}
				\int_{B_r}| D  w-( D  w)_r|\,{\rm d}x\le&C\Big(\frac{r}{R}\Big)^{d+\alpha}\int_{B_R}| D  w-( D  w)_R|\,{\rm d}x \notag\\
				&+C\Big(\frac{r}{R}\Big)^{d+\alpha}\int_{B_R}| D  w- D  h|\,{\rm d}x \notag\\
				&+C\int_{B_R}| D  w- D  h|\,{\rm d}x \notag\\
				\le&C\Big(\frac{r}{R}\Big)^{d+\alpha}\int_{B_R}| D  w-( D  w)_R|\,{\rm d}x \notag\\
				&+C\int_{B_R}| D  w- D  h|\,{\rm d}x. \notag
			\end{align}
			The proof is complete.
		\end{proof}

We now state and prove the main result in this section.

\begin{Theorem}[Gradient regularity in ${\rm BMO}-$spaces]\label{thm_ll}
Let $u\in W^{1,G}_0(\Omega)$ be a weak solution for the transmission problem \eqref{eq_stima118}--\eqref{eq_stima119}. Suppose \eqref{business class}--\eqref{monicavitti} is in force. Then $ D  u\in {\rm BMO}_{\rm loc}(\Omega)$. Moreover, for every $\Omega'\Subset\Omega$, 
\[
	\left\|Du\right\|_{{\rm BMO}(\Omega')}\leq C,
\]
where $C=C(d,\|f\|_{L^\infty(\Gamma)},{\rm diam}(\Omega),{\rm dist}(\Omega',\partial\Omega))>0$.
\end{Theorem}
\begin{proof}	Let $x^0\in\Gamma$, and let $R>0$ such that $B_R:=B(x^0,R)\Subset\Omega$. Let $h\in W_u^{1,G}(B_R)$ be the weak solution of
				\begin{equation*}
					\textnormal{div}\bigg(\frac{g(| D  h|)}{| D  h|} D  h\bigg)=0\quad\textnormal{in }B_R.
				\end{equation*}
				Since $h=u$ on $\partial B_R$ in the trace sense, we can extend $h$ to $\Omega\setminus B_R$ so that $h=u$ in $\Omega\setminus B_R$. This implies that $h\in W_0^{1,G}(\Omega)$ and hence, since $u$ is a global minimizer of \eqref{stima117}, we have
				\begin{align}\label{stima144}
					\int_{\Omega}G(| D  u|)\,{\rm d}x-\int_{\Gamma}fu\,{\rm d}\mathcal{H}^{d-1}\le\int_{\Omega}G(| D  h|)\,{\rm d}x-\int_{\Gamma}fh\,{\rm d}\mathcal{H}^{d-1}.
				\end{align}
				Set $\Gamma_R=B_R\cap\Gamma$. Since $h=u$ in $\Omega\setminus B_R$, \eqref{stima144} becomes
				\begin{align}
					\int_{B_R}G(| D  u|)\,{\rm d}x-\int_{\Gamma_R}fu\,{\rm d}\mathcal{H}^{d-1}\le\int_{B_R}G(| D  h|)\,{\rm d}x-\int_{\Gamma_R}fh\,{\rm d}\mathcal{H}^{d-1} \notag
				\end{align}
				from which, applying the Trace Theorem and Poincaré Inequality, follows
				\begin{align}\label{stima145}
					\int_{B_R}G(| D  u|)\,{\rm d}x-\int_{B_R}G(| D  h|)\,{\rm d}x\le&\int_{\Gamma_R}fu\,{\rm d}\mathcal{H}^{d-1}-\int_{\Gamma_R}fh\,{\rm d}\mathcal{H}^{d-1} \notag\\
					\le&\|f\|_{L^{\infty}(\Gamma)}\int_{\Gamma_R}|u-h|\,{\rm d}\mathcal{H}^{d-1} \notag\\
					\le&C\int_{B_R}|u-h|\,{\rm d}x+C\int_{B_R}| D (u-h)|\,{\rm d}x \notag\\
					\le&C\int_{B_{R}}| D (u-h)|\,{\rm d}x.
					%\le&CR^{\frac{d}{p'}}\| D (u-h)\|_{L^p(B_R)}.
				\end{align}
				From Lemma \ref{stima146}, we bound the left-hand side of \eqref{stima145}
				\begin{align}\label{stima147}
					\int_{B_R}G(| D  u|)\,{\rm d}x-\int_{B_R}G(| D  h|)\,{\rm d}x\ge&C\int_{B_R}| D  (u-h)|^p\,{\rm d}x,
					%\notag\\
				                %=&C\| D (u-h)\|_{L^p(B_R)}^p.
				\end{align}
				and, combining \eqref{stima145} with \eqref{stima147}, we get
				\begin{equation*}
					\int_{B_R}| D  (u-h)|^p\,{\rm d}x \leq C\int_{B_{R}}| D (u-h)|\,{\rm d}x.
					%\| D (u-h)\|_{L^p(B_R)}\le CR^{\frac{d}{p}}
				\end{equation*}
				Using this and H\"older's inequality, we obtain
				\begin{eqnarray*}
					\left( \int_{B_R}| D  (u-h)|\,{\rm d}x \right)^p & \leq & C' R^{d(p-1)}\int_{B_{R}}| D (u-h)|^p\,{\rm d}x\\
					&  \leq & C R^{d(p-1)}\int_{B_{R}}| D (u-h)|\,{\rm d}x 
				\end{eqnarray*}
				and thus 
				\begin{equation*}
					\int_{B_R}| D  (u-h)|\,{\rm d}x \leq C R^d.
				\end{equation*}
				From Proposition \ref{stima161}, we get
				\begin{equation*}
					\int_{B_r}| D  u-( D  u)_r|\,{\rm d}x\le C\Big(\frac{r}{R}\Big)^{d+\alpha}\int_{B_R}| D  u-( D  u)_R|\,{\rm d}x+CR^d\
				\end{equation*}
				for every $0<r\le R$, and, applying Lemma \ref{lem_stima119}, we conclude
				\begin{equation*}
					\int_{B_r}| D  u-( D  u)_r|\,{\rm d}x\le Cr^d, \quad\forall r\in(0,R].
				\end{equation*}
				The proof is complete.
			\end{proof}

\begin{Remark}[Potential estimates and the $p$-Laplace operator]
If $g(t):=t^{p-1}$, the conclusion of Theorem \ref{thm_ll} has been obtained through the use of potential estimates; see \cite[Corollary 1, item (C9)]{KM2014c}. Indeed, notice that for $B_r\subset\Omega$, we have
\[
	\int_{B_r}f{\rm d}\mathcal{H}^{d-1}\leq Cr^{d-1},
\]
which is precisely the condition in \cite[Corollary 1, item (C9)]{KM2014c}. See also \cite{M2011,M2011a}.
\end{Remark}

As a corollary to Theorem \ref{thm_ll}, we obtain a modulus of continuity for the solution $u$ in $C^{0,{\rm Log-Lip}}-$spaces. 

\begin{Corollary}[Log-Lipschitz continuity estimates]\label{cor_ll}
Let $u\in  W_0^{1,G}(\Omega)$ be a weak solution for \eqref{eq_stima118}-\eqref{eq_stima119}. Suppose \eqref{business class}--\eqref{monicavitti} are in force. Then $u\in C^{0,{\rm Log-Lip}}_{\rm loc}(\Omega)$. Moreover, for every $\Omega'\Subset\Omega$, 
\[
	\left\|u\right\|_{C^{0,{\rm Log-Lip}}(\Omega')}\leq C\left(\left\|u\right\|_{L^\infty(\Omega)}+\left\|f\right\|_{L^\infty(\Gamma)}\right),
\]
where $C=C(p,d,{\rm diam}(\Omega),{\rm dist}(\Omega',\partial\Omega))>0$.
\end{Corollary}

Indeed, a function whose partial derivatives are in ${\rm BMO}$ belongs to the Zygmund class (cf. \cite{Zygmund2002}). Because functions in the latter have a $C^{0,{\rm Log-Lip}}$ modulus of continuity, the corollary follows. An alternative argument follows from embedding results for borderline spaces; see \cite[Theorem 3]{Cianchi1996}.

\subsection{H\"older continuity of weak solutions}

Here, we consider unbounded interface data. We work under the condition $f\in W^{1,p'+\varepsilon}(\Omega)$, where $\varepsilon>0$ depends on $p$ and the dimension, and prove a regularity result in H\"older spaces for the weak solutions of \eqref{eq_stima118}-\eqref{eq_stima119}.

\begin{Theorem}\label{thm_c0alpha}
Let $u$ be a weak solution to the interface problem \eqref{eq_stima118}--\eqref{eq_stima119}, under assumptions \eqref{business class}--\eqref{monicavitti}. Let $2<p<d$ and $\varepsilon>0$ be such that
\[
	\frac{d-p}{p-1}<\varepsilon<d-\frac{p}{p-1},
\]
and suppose $f\in W^{1,p'+\varepsilon}(\Omega)$. Then $u\in C^{0,\alpha}_{\textnormal{loc}}(\Omega)$, where 
\[
	\alpha=1-\frac{d}{p+\varepsilon(p-1)},
\]
with estimates.
\end{Theorem}
\begin{proof}We split the proof into three steps. 

\medskip

\noindent{\bf Step 1 - }Combining \eqref{stima145} and \eqref{stima147}, one obtains
				\begin{equation}\label{stima155}
					\| D (u-h)\|_{L^p(B_R)}^p\le C\int_{\Gamma}|f(u-h)|\,{\rm d}\mathcal{H}^{d-1}.
				\end{equation}
				We proceed by examining the right-hand side of \eqref{stima155}. Using the Trace Theorem, we get
				\begin{align}\label{stima156}
					\int_{\Gamma_R}|f(u-h)|\,{\rm d}\mathcal{H}^{d-1}\le&C\int_{B_{R}}|f(u-h)|\,{\rm d}x+C\int_{B_{R}}| D \big(f(u-h)\big)|\,{\rm d}x \notag\\
					\le&C\int_{B_{R}}|f||u-h|\,{\rm d}x+C\int_{B_{R}}| D  f||u-h|\,{\rm d}x \notag \\
					&+C\int_{B_{R}}|f|| D (u-h)|\,{\rm d}x \notag\\
					=:&I_1+I_2+I_3.
				\end{align}
				Now, we estimate each of the summands $I_1$, $I_2$ and $I_3$. Concerning $I_1$, we have
				\begin{align}\label{stima157}
					\int_{B_{R}}|f||u-h|\,{\rm d}x\le&\bigg(\int_{B_{R}}|f|^{p'+\varepsilon}\,{\rm d}x\bigg)^{\frac{1}{p'+\varepsilon}}\bigg(\int_{B_{R}}|u-h|^{\frac{p'+\varepsilon}{p'+\varepsilon-1}}\,{\rm d}x\bigg)^{\frac{p'+\varepsilon-1}{p'+\varepsilon}} \notag\\
					\le&C\Bigg[\bigg(\int_{B_R}|u-h|^{\frac{p'+\varepsilon}{p'+\varepsilon-1}\frac{p'+\varepsilon-1}{p'+\varepsilon}p}\,{\rm d}x\bigg)^{\frac{p'+\varepsilon}{p'+\varepsilon-1}\frac{1}{p}} \notag\\
					&\times\bigg(\int_{B_{R}}\,{\rm d}x\bigg)^{\frac{\frac{p'+\varepsilon-1}{p'+\varepsilon}p-1}{\frac{p'+\varepsilon-1}{p'+\varepsilon}p}}\Bigg]^{\frac{p'+\varepsilon-1}{p'+\varepsilon}} \notag\\
					\le&CR^{d\frac{\frac{p'+\varepsilon-1}{p'+\varepsilon}p-1}{p}}\|u-h\|_{L^p(B_R)} \notag\\
					\le&CR^{d\frac{\frac{p'+\varepsilon-1}{p'+\varepsilon}p-1}{p}}\| D (u-h)\|_{L^p(B_R)}.
				\end{align}
				To estimate $I_2$, one notices that
				\begin{align}\label{stima158}
					\int_{B_{R}}| D  f||u-h|\,{\rm d}x\le&\bigg(\int_{B_{R}}| D  f|^{p'+\varepsilon}\,{\rm d}x\bigg)^{\frac{1}{p'+\varepsilon}}\bigg(\int_{B_{R}}|u-h|^{\frac{p'+\varepsilon}{p'+\varepsilon-1}}\,{\rm d}x\bigg)^{\frac{p'+\varepsilon-1}{p'+\varepsilon}} \notag\\
					\le&C\Bigg[\bigg(\int_{B_R}|u-h|^{\frac{p'+\varepsilon}{p'+\varepsilon-1}\frac{p'+\varepsilon-1}{p'+\varepsilon}p}\,{\rm d}x\bigg)^{\frac{p'+\varepsilon}{p'+\varepsilon-1}\frac{1}{p}} \notag\\
					&\times\bigg(\int_{B_{R}}\,{\rm d}x\bigg)^{\frac{\frac{p'+\varepsilon-1}{p'+\varepsilon}p-1}{\frac{p'+\varepsilon-1}{p'+\varepsilon}p}}\Bigg]^{\frac{p'+\varepsilon-1}{p'+\varepsilon}} \notag\\
					\le&CR^{d\frac{\frac{p'+\varepsilon-1}{p'+\varepsilon}p-1}{p}}\|u-h\|_{L^p(B_R)} \notag\\
					\le&CR^{d\frac{\frac{p'+\varepsilon-1}{p'+\varepsilon}p-1}{p}}\| D (u-h)\|_{L^p(B_R)}.
				\end{align}
				Finally, we examine $I_3$. Indeed, 
				\begin{align}\label{stima159}
					\int_{B_{R}}|f|| D (u-h)|\,{\rm d}x\le&\bigg(\int_{B_{R}}|f|^{p'+\varepsilon}\,{\rm d}x\bigg)^{\frac{1}{p'+\varepsilon}}\\
					& \times\bigg(\int_{B_{R}}| D (u-h)|^{\frac{p'+\varepsilon}{p'+\varepsilon-1}}\,{\rm d}x\bigg)^{\frac{p'+\varepsilon-1}{p'+\varepsilon}} \notag\\
					\le&C\Bigg[\bigg(\int_{B_R}| D (u-h)|^{\frac{p'+\varepsilon}{p'+\varepsilon-1}\frac{p'+\varepsilon-1}{p'+\varepsilon}p}\,{\rm d}x\bigg)^{\frac{p'+\varepsilon}{p'+\varepsilon-1}\frac{1}{p}} \notag\\
					&\times\bigg(\int_{B_{R}}\,{\rm d}x\bigg)^{\frac{\frac{p'+\varepsilon-1}{p'+\varepsilon}p-1}{\frac{p'+\varepsilon-1}{p'+\varepsilon}p}}\Bigg]^{\frac{p'+\varepsilon-1}{p'+\varepsilon}} \notag\\
					\le&CR^{d\frac{\frac{p'+\varepsilon-1}{p'+\varepsilon}p-1}{p}}\| D (u-h)\|_{L^p(B_R)}.
				\end{align}
				Because of the role played by the exponents in the previous inequalities, we conclude this step by noticing that
				\begin{align*}
					\frac{\frac{p'+\varepsilon-1}{p'+\varepsilon}p-1}{p}
					%=&\frac{(p'+\varepsilon-1)p-p'-\varepsilon}{p(p'+\varepsilon)}\\
					%=&\frac{pp'+p\varepsilon-p-p'-\varepsilon}{p(p'+\varepsilon)}\\
					%=&\frac{p+p'+p\varepsilon-p-p'-\varepsilon}{p(p'+\varepsilon)}\\
					=&\frac{\varepsilon(p-1)}{p(p'+\varepsilon)}.
				\end{align*}
				
\medskip

\noindent{\bf Step 2 -}Now we combine \eqref{stima155}, \eqref{stima156}, \eqref{stima157}, \eqref{stima158}, and \eqref{stima159} to produce
				\begin{equation*}
					\| D (u-h)\|_{L^p(B_R)}^{p-1}\le CR^{d\frac{\varepsilon(p-1)}{p(p'+\varepsilon)}}.
				\end{equation*}
				As a consequence, it follows that 
				\begin{equation}\label{stima160}
					\int_{B_{R}}| D (u-h)|^p\,{\rm d}x\le CR^{d\frac{\varepsilon}{p'+\varepsilon}}.
				\end{equation}
				Hence, Proposition \ref{stima161} builds upon \eqref{stima160} to yield
				\begin{equation*}
					\int_{B_{r}}| D  u-( D  u)_r|\,{\rm d}x\le C\Big(\frac{r}{R}\Big)^{d+\alpha}\int_{B_{R}}| D  u-( D  u)_R|\,{\rm d}x+CR^{d\frac{\varepsilon}{p'+\varepsilon}}.
				\end{equation*}
				The former inequality, together with Lemma \ref{lem_stima119}, leads to
				\begin{equation*}
					\int_{B_{r}}| D  u-( D  u)_r|\,{\rm d}x\le Cr^{d\frac{\varepsilon}{p'+\varepsilon}}\quad\forall r\in(0,R],
				\end{equation*}
				and one easily concludes
				\begin{equation}\label{doron}
					r^{(d-d\frac{\varepsilon}{p'+\varepsilon})-d}\int_{B_{r}}| D  u-( D  u)_r|^p\,{\rm d}x\le C\quad\forall r\in(0,R].
				\end{equation}
\medskip

\noindent{\bf Step 3 - }The inequality in \eqref{doron} implies  $Du\in L_C^{p,\lambda}(\Omega;\mathbb{R}^d)$, with
				\begin{equation*}
					\lambda:=d\bigg(1-\frac{\varepsilon}{p'+\varepsilon}\bigg)=\frac{dp}{p+p\varepsilon-\varepsilon}.
				\end{equation*}
				Since $\lambda<d$, we have $L_C^{p,\lambda}(\Omega;\mathbb{R}^d)=L_M^{p,\lambda}(\Omega;\mathbb{R}^d)$; as a consequence $u\in C_{\textnormal{loc}}^{0,\alpha}(\Omega)$ with
				\begin{equation*}
					\alpha=1-\frac{\lambda}{p}
				\end{equation*}
				if $p>\lambda$. That is, if
				\begin{equation*}
					\varepsilon>\frac{d-p}{p-1},
				\end{equation*}
				which holds by assumption.
				Since $p'+\varepsilon<d$, we finally get
				\begin{equation*}
					\frac{d-p}{p-1}<\varepsilon<d-\frac{p}{p-1}.
				\end{equation*}
\end{proof}

\begin{Remark}[Endpoint-regularity]We conclude by examining the limit behaviour of the modulus of continuity -- encoded by the H\"older exponent $\alpha\in(0,1)$ in Theorem \ref{thm_c0alpha} -- as $\varepsilon$ approaches the endpoints of its interval of definition. Indeed, as
\[
	\varepsilon\to\bigg(d-\frac{p}{p-1}\bigg)^-
\]
one gets	
\[
	\alpha\to1-\frac{1}{p-1}.
\]
On the other hand, as
\[
	\varepsilon\to\bigg(\frac{d-p}{p-1}\bigg)^+
\]
one has
\[
	\alpha\to0.
\]
\end{Remark}

\bigskip

{\small \noindent{\bf Acknowledgments.} The authors thank Paolo Baroni and Giuseppe Mingione for insightful comments on the material in the paper. VB is supported by the Centre for Mathematics of the University of Coimbra (UIDB/00324/2020, funded by the Portuguese Government through FCT/MCTES). EP is partially supported by the Centre for Mathematics of the University of Coimbra (UIDB/00324/2020, funded by the Portuguese Government through FCT/MCTES) and by FAPERJ (grants E26/200.002/2018 and E26/201.390/2021). JMU is partially supported by the King Abdullah University of Science and Technology (KAUST) and by the Centre for Mathematics of the University of Coimbra (UIDB/00324/2020, funded by the Portuguese Government through FCT/MCTES).}

\end{document}